\documentclass[graybox]{svmult}
\usepackage{mathptmx}       
\usepackage{helvet}         
\usepackage{courier}        
\usepackage{makeidx}         
\usepackage{graphicx}        
\usepackage{multicol}        
\usepackage[bottom]{footmisc}
\usepackage{amsmath,amssymb,amsfonts, color}      
\usepackage{tikz} 
\usepackage{wrapfig}
\usetikzlibrary{arrows,shapes,positioning}
\usetikzlibrary{decorations.markings}
\tikzstyle arrowstyle=[scale=1]
\tikzstyle directed=[postaction={decorate,decoration={markings,
    mark=at position .65 with {\arrow[arrowstyle]{stealth}}}}]
\tikzstyle reverse directed=[postaction={decorate,decoration={markings,
    mark=at position .65 with {\arrowreversed[arrowstyle]{stealth};}}}]
    \usepackage{float}

\usepackage[skip=2pt,font=scriptsize]{caption}

\newcommand{\dist}{\operatorname{dist}}

\newcommand{\pa}{\partial}
\newcommand{\eps}{\varepsilon}

\newcommand{\tr}{{\rm tr}}

\newcommand{\cC}{\mathcal{C}}
\newcommand{\cH}{\mathcal{H}}

\newcommand{\cN}{{\mathcal{N}}}

\newcommand{\N}{\mathbb{N}}

\newcommand{\R}{\mathbb{R}}

\newcommand{\cL}{\mathcal{L}}
\renewcommand{\epsilon}{\varepsilon}

        \newcommand{\erfc}{\operatorname{erfc} }
        \newcommand{\erf}{\operatorname{erf} }
\newcommand{\sff}{\mathbb{I}}

\makeindex            

\begin{document}
\title*{How to hear the corners of a drum} 

\author{Medet Nursultanov, Julie Rowlett, and David Sher}

\institute{Medet Nursultanov \at Mathematical Sciences, Chalmers University of Technology and University of Gothenburg, 412 96 Gothenburg, Sweden, Institute of mathematics and mathematical modeling, Astana, Kazakhstan \email{medet@chalmers.se}
\and Julie Rowlett \at Mathematical Sciences, Chalmers University of Technology, 412 96 Gothenburg, Sweden, \email{julie.rowlett@chalmers.se}
\and David Sher \at Department of Mathematical Sciences, DePaul University, 2320 N Kenmore Ave, Chicago, IL 60614, USA \email{dsher@depaul.edu} 
}
\maketitle

\abstract{We announce a new result which shows that under either Dirichlet, Neumann, or Robin boundary conditions, the corners in a planar domain are a spectral invariant of the Laplacian.  For the case of polygonal domains, we show how a locality principle, in the spirit of Kac's ``principle of not feeling the boundary'' can be used together with calculations of explicit model heat kernels to prove the result.  In the process, we prove this locality principle for all three boundary conditions.  Albeit previously known for Dirichlet boundary conditions, this appears to be new for Robin and Neumann boundary conditions, in the generality presented here. For the case of curvilinear polygons, we describe how the same arguments using the locality principle fail, but can nonetheless be replaced by powerful microlocal analysis methods.}

\section{Introduction}
\label{intro}

\begin{figure}[b]

\includegraphics[scale=.65]{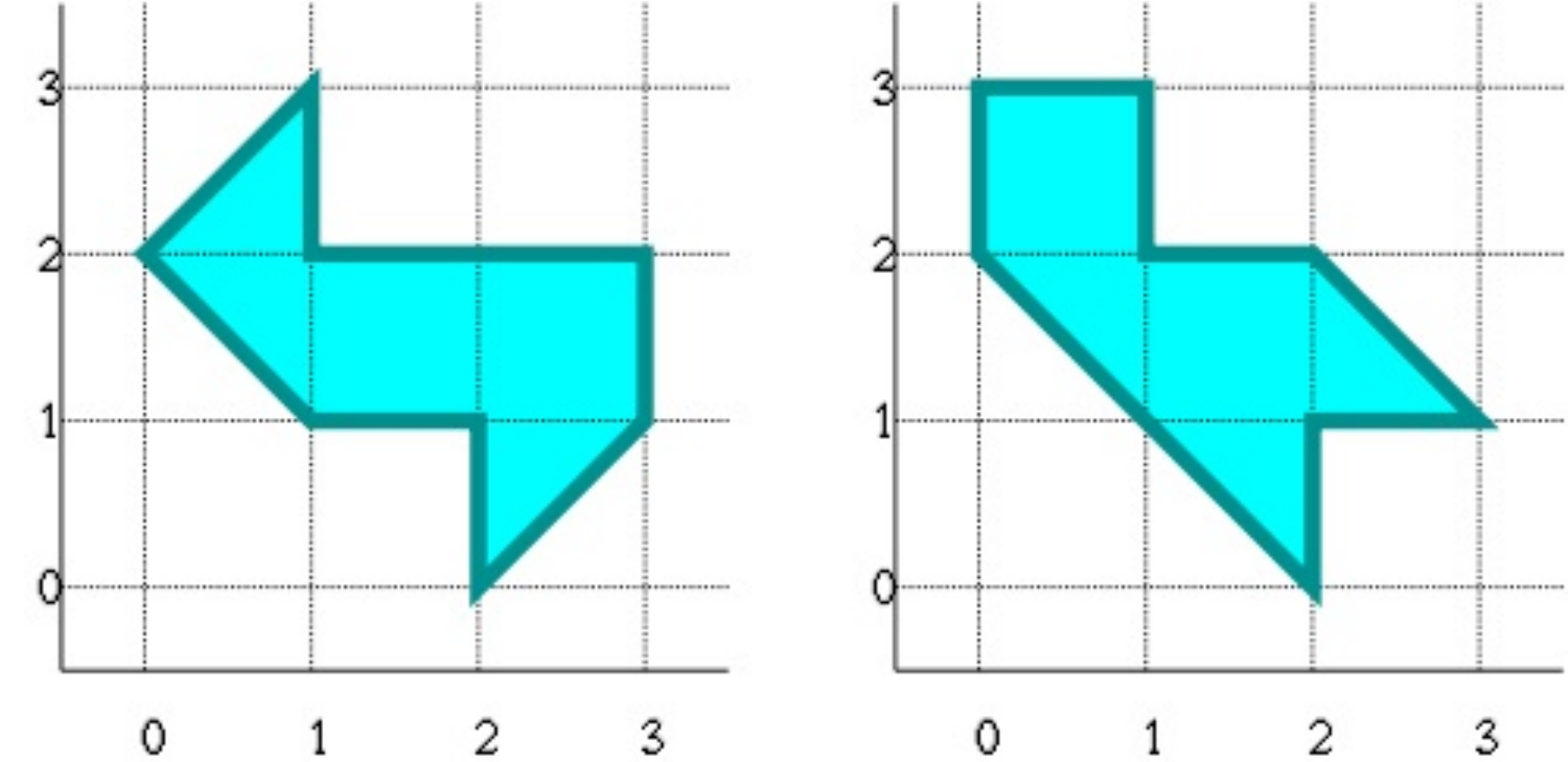}

\caption{These two domains were demonstrated by Gordon, Webb, and Wolpert to be isospectral for the Laplacian with Dirichlet boundary condition \cite{gww2}.  This image is from Wikipedia Commons.}  
\label{fig1}      
\end{figure}

It is well known that ``one cannot hear the shape of a drum'' \cite{gww1}.  Mathematically, this means that there exist bounded planar domains which have the same eigenvalues for the Laplacian with Dirichlet boundary condition, in spite of the domains having different shapes.  The standard example is shown in Figure \ref{fig1}.   Two geometric characteristics of these domains are immediately apparent:  
\begin{enumerate}
\item These domains both have corners. 
\item Neither of these domains are convex.
\end{enumerate}
This naturally leads to the following two open problems:
\begin{problem}  \label{prob1}
Can one hear the shape of a smoothly bounded drum?
\end{problem}
\begin{problem}  \label{prob2} 
Can one hear the shape of a convex drum?
\end{problem} 
The mathematical formulation of these problems are:  if two smoothly bounded (respectively, convex) domains in the plane are isospectral for the Laplacian with Dirichlet boundary condition, then are they the same shape?  

One could dare to conjecture that the answer to Problem \ref{prob1} is yes, based on the isospectrality result of Zelditch \cite{zelda}.  He proved that under certain geometric and symmetry conditions, if two analytically bounded domains are isospectral, then they have the same shape.  For certain classes of convex polygonal domains including triangles \cite{durso}, \cite{gm}; parallelograms \cite{sos}; and trapezoids \cite{hlr};  if two such domains are isospectral, then they are indeed the same shape.  This could lead one to suppose that perhaps Problem \ref{prob2} also has a positive answer.  

Contemplating these questions lead the second author and Z. Lu to investigate whether smoothly bounded domains can be isospectral to domains with corners.  In \cite{corners}, they proved that for the Dirichlet boundary condition, ``one can hear the corners of a drum'' in the sense that a domain with corners cannot be isospectral to a smoothly bounded domain.  In forthcoming work, the authors of the present paper shall generalize that result to both Neumann and Robin boundary conditions.  The first purpose of this paper is to announce those results and to describe the key ideas of the proofs.  

The second purpose of this paper is to provide a proof of the locality principle for the Neumann and Robin boundary conditions in a general context which includes domains with only piecewise smooth boundary. This is a generalization of Kac's ``principle of not feeling the boundary'' \cite{kac}.  To explain this principle, let $\Omega$ be a bounded domain in $\R^2$, or more generally $\R^n$, because the argument works in the same way in all dimensions.  Assume the Dirichlet boundary condition, and let the corresponding heat kernel for $\Omega$ be denoted by $H$, while the heat kernel for $\R^n$, 
\begin{equation} \label{hkrn} K(z, z', t) = (4\pi t)^{-n/2} e^{-d(z, z')^2/4t}. \end{equation} 
Let 
$$\delta = \min \{ d(z, \pa \Omega), d(z', \pa \Omega) \}.$$
Then, there are constants $A, B > 0$ such that 
$$| K(z,z',t) - H(z,z',t)| \leq A t^{-n/2} e^{-B \delta^2 / t}.$$
This means that the heat kernel for $\Omega$ is $O(t^\infty)$\footnote{By $O(t^\infty)$, we mean $O(t^N)$ for any $N \in \N$. } close to the Euclidean heat kernel, as long as we consider points $z, z'$ which are at a positive distance from the boundary; hence the heat kernel does not ``feel the boundary.''   

In a similar spirit, a more general locality principle is known to be true.  The idea is that one has a collection of sets which are ``exact geometric matches'' to certain pieces of the domain $\Omega$.  To describe the meaning of an ``exact geometric match,''  consider a piece of the first quadrant near the origin in $\R^2$.  A sufficiently small piece is an exact match for a piece of a rectangle near a corner.  Similarly, for a surface with exact conical singularities, near a singularity of opening angle $\gamma$, a piece of an infinite cone with the same opening angle is an exact geometric match to a piece of the surface near that singularity.  For a planar example, see Figure \ref{fig2}.   The locality principle states that if one takes the heat kernels for those ``exact geometric matches,'' and restricts them to the corresponding pieces of the domain (or manifold), $\Omega$, then those ``model heat kernels'' are equal to the heat kernel for $\Omega$, restricted to the corresponding pieces of $\Omega$, with error $O(t^\infty)$ as $t \downarrow 0$.  

This locality principle is incredibly useful, because if one has exact geometric matches for which one can explicitly compute the heat kernel, then one can use these to compute the short time asymptotic expansion of the heat trace.  Moreover, in addition to being able to compute the heat trace expansion, one can also use this locality principle to compute the zeta regularized determinant of the Laplacian as in \cite{polyakov}.  

Here, we shall explain one application of the locality principle:  ``how to hear the corners of a drum.''  
\begin{theorem} \label{thiso} Let $\Omega \subset \R^2$ be a simply connected, possibly curvilinear, polygonal domain.  Then, the Laplacian with either Dirichlet, Neumann, or Robin boundary condition is not isospectral to the Laplacian with the same boundary condition\footnote{In particular, in the case of Robin boundary conditions, we assume the same Robin parameters for both domains.}   on any smoothly bounded domain.   
\end{theorem}  

We shall not give a complete proof of this result, but instead, we shall explain the main ideas and ingredients.  In the case of classical polygonal domains whose boundary consists of straight line segments, one can use the locality principle together with the heat kernels for the ``exact geometric matches'' to prove Theorem \ref{thiso}.  This locality principle is stated and proven in \S \ref{sloc}.  The model heat kernels as well as the corresponding Green's functions for the ``exact geometric matches'' used in the proof of Theorem \ref{thiso} are given in \S \ref{hearing}.  We outline how to use these model heat kernels together with the locality principle to prove Theorem \ref{thiso} for the case of classical (straight-edged) polygonal domains in \S \ref{hearing}.  In conclusion, we explain in how the locality principle \em fails \em to prove Theorem \ref{thiso} for the case of curvilinear polygonal domains.  This motivates the discussion in \S 4 concerning the necessity and utility of microlocal analysis, in particular, the construction of the heat kernel for curvilinear polygonal domains via the heat space and heat calculus in that setting.  This construction, together with the proof of Theorem \ref{thiso} in full detail, is currently in preparation, and shall be presented in forthcoming work.  

\section{The locality principle}  \label{sloc}

We begin by setting notations and sign conventions and recalling fundamental concepts.  

\subsection{Geometric and analytic preliminaries} 
\label{prelim} 
To state the locality principle, we make the notion of an ``exact geometric match'' precise.  Let $\Omega$ be a domain, possibly infinite, contained in $\R^n$.   

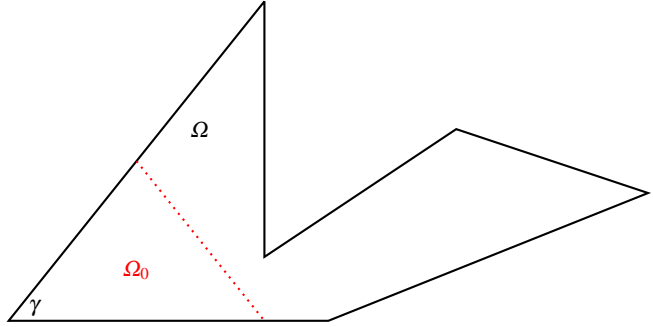
\begin{figure}[H] 
\begin{minipage}[b][4.75cm][s]{.45\textwidth}
\begin{tikzpicture}[scale=.85]

\draw[thick] (4,5)--(0,0)--(5,0)--(10,2)--(7,3)--(4,1)--(4,5);

 \node at (.40, .20) [align=center]{$\gamma$};
 
  \node at (3, 3) [align=center]{$\Omega$};

\draw[red, thick, dotted] (2,2.5)--(4,0);

 \node[red] at (2, 0.8) [align=center]{$\Omega_0$};

 \end{tikzpicture} 
\end{minipage} 
\caption{Above, we have the polygonal domain $\Omega$ which contains the triangular domain, $\Omega_0$.  Letting $S=S_\gamma$ be a circular sector of opening angle $\gamma$ and infinite radius, this is an example of an ``exact geometric match,'' in the sense that $\Omega_0$ is equal to a piece of $S$. } \label{fig2} 
\end{figure}

\begin{definition} \label{exactmatch} Assume that $\Omega_0 \subset \Omega \subset \R^n$, and $S \subset \R^n$.  We say that $S$ and $\Omega$ are \em exact geometric matches on $\Omega_0$ \em if there exists a sub-domain $\Omega_c\subseteq\Omega$ which compactly contains $\Omega_0$ and which is isometric to a sub-domain of $S$ (which, abusing notation, we also call $\Omega_c$). Recall that $\Omega_0$ being compactly contained in $\Omega_c$ means that the distance from $\overline{\Omega}_0$ to $\overline{\Omega\setminus \Omega_c}$ is positive. 
\end{definition}

Next, we recall the heat kernel in this context.  The heat kernel, $H$, is the Schwartz kernel of the fundamental solution of the heat equation.  It is therefore defined on $\Omega \times \Omega \times [0, \infty)$, and satisfies 
$$H(z, z', t) = H(z', z, t), \quad (\pa_t + \Delta)H (z, z', t)  = 0 \textrm{ for } t > 0,$$
$$H(z, z', 0) = \delta(z-z'), \quad \textrm{ in the distributional sense.}$$
Throughout we use the sign convention for the Laplacian, $\Delta$, on $\R^n$, that
$$\Delta = - \sum_{j=1} ^n \pa_j ^2.$$
We consider three boundary conditions in the statement of Theorem 1.  These are: 
\begin{enumerate} 
\item[(D)] the \em Dirichlet boundary condition, \em which requires the function to vanish on the boundary; 
\item[(N)] the \em Neumann boundary condition, \em which requires the normal derivative of the function to vanish on the boundary; 
\item[(R)] the \em Robin boundary condition, \em which requires the function, $u$, to satisfy the following equation on the boundary
\begin{equation} \label{rbc1} \alpha u + \beta \frac{\pa u}{\pa \nu} = 0, \quad \frac{\pa u}{\pa \nu} \textrm{ is the outward pointing normal derivative.} \end{equation} 
\end{enumerate} 

For $u \in \cL^2 (\Omega)$, the heat equation with initial data given by $u$ is then solved by 
$$\int_\Omega H(z, z', t) u(z') dz'.$$
Moreover, if $\Omega$ is a bounded domain, and $\{ \phi_k \}_{k \geq 1}$ is an orthonormal basis for $\cL^2(\Omega)$ consisting of eigenfunctions of the Laplacian satisfying the appropriate boundary condition, with corresponding eigenvalues $\{ \lambda_k \}_{k \geq 1}$, then the heat kernel 
$$H(z, z', t) = \sum_{k \geq 1} e^{-\lambda_k t} \phi_k (z) \phi_k (z').$$

\subsection{Locality principle for Dirichlet boundary condition}
In the general context of domains in $\R^n$ which have only piecewise smooth boundary, the key point is that the locality principle should hold \em up to the boundary. \em  This differs from many previous presentations of a locality principle. For example, in \cite[Theorem 1.1]{list}, it is proved that without any condition on the regularity of the boundary, for any choice of self-adjoint extension of the Laplacian on $\Omega \subset \R^n$, the heat kernel for this self adjoint extension of the Laplacian on $\Omega$, denoted by $H^\Omega$ satisfies 
$$|H^\Omega (x,y,t) - H^0 (x,y,t)| \leq (C_a \rho(x,y)^{-n} + C_b) \cdot \frac{ \exp \left( - \frac{ (\rho(x) + \rho(y))^2}{4t} \right) }{t^{2 \lceil \frac{n+1}{2} \rceil - \frac 1 2 }}.$$
Above, $H^0$ is the heat kernel for $\R^n$, $\rho(x) = \dist(x, \pa \Omega)$, $\rho(x,y) = \min( \rho(x), \rho(y))$.  The constants $C_a$ and $C_b$ can also be calculated explicitly according to \cite{list}.  Clearly, the estimate loses its utility as one approaches the boundary.  

In the case of smoothly bounded domains, there is a result of L\"uck \& Schick \cite[Theorem 2.26]{luck}, which implies the locality principle for Dirichlet (and Neumann) boundary condition, and which holds all the way up to the boundary.  We recall that result.  

\begin{theorem}[L\"uck \& Schick]  Let $N$ be a Riemannian manifold possibly with boundary which is of bounded geometry.  Let $V \subset N$ be a closed subset which carries the structure of a Riemannian manifold of the same dimension as $N$ such that the inclusion of $V$ into $N$ is a smooth map respecting the Riemannian metrics.  (We make no assumptions about the boundaries of $N$ and $V$ and how they intersect.)  For fixed $p \geq 0$, let $\Delta[V]$ and $\Delta[N]$ be the Laplacians on $p$-forms on $V$ and $N$, considered as unbounded operators with either absolute boundary conditions or with relative boundary conditions (see Definition 2.2 of \cite{luck}).  Let $\Delta[V]^k e^{-t \Delta[V]} (x,y)$ and $\Delta[N]^k e^{-t \Delta[N]} (x,y)$ be the corresponding smooth integral kernels.  Let $k$ be a non-negative integer.

Then there is a monotone decreasing function $C_k (K): (0, \infty) \to (0, \infty)$ which depends only on the geometry of $N$ (but not on $V$, $x$, $y$, $t$) and a constant $C_2$ depending only on the dimension of $N$ such that for all $K>0$ and $x,y \in V$ with $d_V (x) := d(x, N\setminus V) \geq K$, $d_V(y) \geq K$ and all $t>0$:
$$\left| \Delta[V]^k e^{-t \Delta[V]} (x,y) - \Delta[N]^k e^{-t \Delta[N]} (x,y) \right| \leq C_k (K) e^{- \left( \frac{d_V(x)^2 + d_V(y)^2 + d(x,y)^2}{C_2 t} \right)}.$$
\end{theorem}

One may therefore compare the heat kernels for the Laplacian acting on functions, noting (see p. 362 of \cite{taylor}) that relative boundary conditions are Dirichlet boundary conditions, and absolute boundary conditions are Neumann boundary conditions.  We present this as a corollary to L\"uck and Schick's theorem.
\begin{corollary}
Assume that $S$ is an exact match for a piece of $\Omega$, for two smoothly bounded domains in $\R^n$.  Assume the same boundary condition, either Dirichlet or Neumann, for the Euclidean Laplacian on both domains.  Then 
$$\left| H^\Omega (z, z', t) - H^S (z, z', t) \right| = O(t^\infty) \textrm{ as } t \downarrow 0, \quad \textrm{ uniformly for } z, z' \in \Omega_0.$$
\end{corollary} 
\begin{proof} 
We use the theorem of L\"uck and Schick twice, once with $N=\Omega$ and once with $N=S$, with $V=\Omega_c$ in both cases.  We set $k=0$ and 
$$K=\alpha =d(\Omega_0,S\setminus\Omega_c).$$ 
By the assumption of an exact geometric match, $\alpha>0$.  In the $N=S$ case, the theorem reads
\begin{equation*}
	|H^S(t,z,z')-H^{\Omega_c}(t,z,z')| \leq C_0(\alpha)e^{-\frac{|\mathrm{dist}(z,S\setminus\Omega_c)|^2}{C_2 t} - \frac{|\mathrm{dist}(z',S\setminus\Omega_c)|^2}{C_2 t} }\leq C_0(\alpha)e^{-\frac{2 \alpha^2}{C_2 t}}.
	\end{equation*}
We conclude that
	\begin{equation*}
	|H^S(t,z,z')-H^{\Omega_c}(t,z,z')|=O(t^{\infty})
	\end{equation*}
	uniformly on $\Omega_0$. The same statement holds with $S$ replaced by $\Omega$, and then the triangle inequality completes the proof.  
\end{proof} 

Although it may seem quite general, the assumption of smooth boundary is quite restrictive, and the proof in \cite{luck} relies heavily on this assumption.  To the best of our knowledge, the first locality result which holds all the way up to the boundary and includes domains which have only piecewise smooth boundary, but may have for example corners, was demonstrated by van den Berg and Srisatkunarajah \cite{vdbs}.  We note that this result is not stated in the precise form below in \cite{vdbs}, but upon careful reading, it is straightforward to verify that this result is indeed proven there and is also used in several of their calculations.  

\begin{theorem}[van den Berg \& Srisatkunarajah]
Let $\Omega \subset \R^2$ be a polygonal domain.  Let $H^\Omega$ denote the heat kernel for the Laplacian on $\Omega$ with the Dirichlet boundary condition.  Then, for $S=S_\gamma$, a sector of opening angle $\gamma$, and for any corner of $\Omega$ with opening angle $\gamma$, there is a neighborhood $\cN_\gamma$ such that $$|H^\Omega (x,y,t) - H^{S_\gamma} (x,y,t)| = O(t^\infty), \quad \textrm{uniformly} \quad \forall (x,y) \in \cN_\gamma \times \cN_\gamma,$$
Above, $H^{S_\gamma}$ denotes the heat kernel for $S_\gamma$ with the Dirichlet boundary condition.  Moreover, for any $\cN_e \subset \Omega$ which is at a positive distance to all corners of $\Omega$, 
$$|H^\Omega(x,y,t) - H^{\R^2 _+} (x,y,t)| = O(t^\infty), \quad \textrm{uniformly} \quad \forall (x,y) \in \cN_e \times \cN_e.$$
Above, $H^{\R^2 _+}$ denotes the heat kernel for a half space with the Dirichlet boundary condition. 
\end{theorem} 

The proof uses probabilistic methods.  We are currently unaware of a generalization to domains with corners in higher dimensions.  However, it is reasonable to expect that such a generalization holds.  Since Theorem 1 has already been demonstrated for the Dirichlet boundary condition in \cite{corners}, we are interested in the Neumann and Robin boundary conditions.  For this reason, we shall give a proof of a locality principle for both Neumann and Robin boundary conditions which holds in all dimensions, for domains with piecewise smooth boundary (in fact, only piecewise $\cC^3$ boundary is required), as long as we have a suitable estimate on the second fundamental form on the boundary.  Moreover, our locality principle, similar to that of \cite{vdbs}, allows one to compare the heat kernels all the way up to the boundary.  

\subsection{Locality principle for Neumann boundary condition}
Here we prove a locality principle for domains in $\R^n$ with piecewise $\cC^2$ boundary satisfying some relatively general geometric assumptions. The following is a uniform version of an interior cone condition:

\begin{definition} 
Let $\epsilon>0$ and $\delta>0$. We say that a domain $\Omega \subset \R^n$ satisfies the $(\eps,\delta)$-cone condition if, for every $x \in \pa \Omega$, there exists a ball $B(x,\delta)$ centered at $x$ of radius $\delta$, and a direction $\xi_x$, such that for all $y \in B(x, \delta)\cap \Omega$, the cone with vertex $y$ directed by $\xi_x$ of angle $\eps$ is contained in $\Omega$.  
\end{definition}

\begin{definition} 
	Let $\epsilon>0$ and $h>0$. We say that a domain $\Omega \subset \R^n$ satisfies the $(\eps,h)$-cone condition if, for every $x \in \pa \Omega$, there exists a ball $B(x,\delta)$ centered at $x$ of radius $\delta$, and a direction $\xi_x$, such that for all $y \in B(x, \delta)\cap \Omega$, the cone with vertex $y$ directed by $\xi_x$ of angle $\eps$ and height $h$ is contained in $\Omega$.  
\end{definition}

\begin{definition} 
	Let $\epsilon>0$ and $h>0$. We say that a domain $\Omega \subset \R^n$ satisfies the both side $(\eps,h)$-cone condition if both $\Omega$ and $\mathbb{R}^n\setminus\Omega$ satisfy the $(\eps,h)$-cone condition.
\end{definition}

\begin{theorem}[Locality Principle for Neumann Boundary Condition]  \label{thml1} 
Let $\Omega$, $\Omega_0$, and $S$ be domains in $\R^n$ such that $S$ and $\Omega$ are exact geometric matches on $\Omega_0$, as in Definition \ref{exactmatch}.  Assume that both $\Omega$ and $S$ satisfy the $(\eps,\delta)$-cone condition for some $\eps > 0$ and $\delta>0$.  Let $H^\Omega$ denote the heat kernel associated to the Laplacian on $\Omega$, and let $H^{S}$ denote the heat kernel on $S$, with the same boundary condition for $\pa S$ as taken on $\pa \Omega$.  Moreover, assume that there exists $\sigma \in \R$ such that the second fundamental form $\sff \geq - \sigma$ holds on all the $\cC^2$ pieces of $\pa \Omega$ and $\pa S$.  Then
$$\left| H^\Omega (z, z', t) - H^{S} (z, z', t) \right| = O(t^\infty) \textrm{ as } t \downarrow 0, \quad \textrm{ uniformly for } z, z' \in \Omega_0.$$
\end{theorem} 


\begin{proof} 
We use a patchwork parametrix construction, as discussed in section 3.2 of \cite{polyakov}. This is a general technique that works to construct heat kernels whenever we have exact geometric matches for each part of a domain.

Let $\{\chi_j\}_{j=1}^2$ be a $\cC^{\infty}$ partition of unity on $\Omega$. Assume that $\tilde\chi_j \in \cC^{\infty}(\Omega)$ is identically 1 on a small neighborhood of the support of $\chi_j$ and vanishes outside a slightly larger neighborhood. In particular, we choose $\tilde \chi_1$ to be identically equal to $1$ on $\Omega_0$ and to have support equal to $\Omega_c$.  We then define the patchwork heat kernel
\[G(t,z,z'):=\sum_{j=1}^2 \tilde\chi_j(z) H^{S}(t,z,z')\chi_j(z').\]
We claim that in fact, uniformly for all $z$, $z'\in\Omega_0$,
\[|H^{\Omega}(t,z,z')-G(t,z,z')|=O(t^{\infty}), \quad t \downarrow 0.\]
That is, we claim that the patchwork heat kernel is equal to the true heat kernel with an error that is $O(t^{\infty})$ for small time. This claim immediately implies our result, since on $\Omega_0$, $\chi_1=1$, and $\tilde\chi_1=1$, and thus $G(t,z,z')=H^{S}(t,z,z')$.  

To prove the claim, we follow the usual template. Observe that
\[E(t,z,z'):=(\partial_t+\Delta)G(t,z,z)=\sum_{j=1}^{2}[\Delta,\tilde\chi_j(z)]H^{S}(t,z,z')\chi_j(z').\]
Each commutator $[\Delta,\tilde\chi_j(z)]$ is a first-order differential operator with support a positive distance from the support of $\chi_j$. Thus $E(t,z,z')$ is a sum of model heat kernels and their first derivatives, cut off so that their spatial arguments are a positive distance from the diagonal. We claim each such term is $O(t^{\infty})$.  To obtain this estimate, we use \cite[Theorem 1.1]{wang1}, which gives the estimate 
$$| \nabla H^D (x,y,t) | \leq \frac{C_\alpha}{t^{(n+1)/2}} \exp \left( - \frac{|x-y|}{C_\beta t} \right), \quad x,y \in D,$$
for some constants $C_\alpha, C_\beta > 0$, for $D=\Omega$ and $D=S$.  The setting there is not identical, so we note the places in the proof where minor modifications are required.  First, the assumption that $\Omega$ is compact is used there to obtain estimates for \em all $t>0$.  \em  In particular, the discreteness of the spectrum is used to obtain long time estimates by exploiting the first positive eigenvalue in (2.1) of \cite{wang1}.  Since we are only interested in $t \downarrow 0$, this long time estimate is not required.  Next, compactness is used to be able to estimate the volume of balls, $|B(x, \sqrt{t})| \geq C_\eps t^{\frac n 2}$, for a uniform constant $C_\eps$.  However, we have this estimate due to the $(\eps,\delta)$-cone condition which is satisfied for both $\Omega$ and $S$ which are contained in $\R^n$.  Moreover, we have verified \cite{wang5} that the assumption of piecewise $\cC^2$ boundary (rather than $\cC^2$ boundary) is sufficient for the proof of \cite[Theorem 1.1]{wang1}, as well as the references used therein:  \cite{wang2}, \cite{wang3}, \cite{wang4}.\footnote{We have also verified in private communication with F. Y. Wang that the arguments in \cite{wang1}, \cite{wang2}, \cite{wang3}, \cite{wang4} apply equally well under the curvature assumption $\sff \geq - \sigma$ for piecewise $\cC^2$ boundary.}  

Since the domains $S$ and $\Omega$ satisfy the $(\eps,\delta)$-cone condition, there are Gaussian upper bounds for the corresponding Neumann heat kernels. Specifically, as a result of \cite[Theorems 6.1, 4.4]{d}, for any $T>0$, there exist $C_1, C_2>0$ such that 
\begin{equation}\label{est_for_PC_Pomega}
|H^S(t,x,y)|\leq C_1 t^{- \frac{n}{2}} e^{-\frac{|x-y|^2}{C_2t}},
\qquad
|H^{\Omega}(t,x,y)|\leq C_1 t^{- \frac n 2} e^{-\frac{|x-y|^2}{C_2t}}
\end{equation}
on $(0,T]\times S\times S$ and $(0,T]\times \Omega\times \Omega$ respectively.  The upshot is that each term in the sum defining $E(t,z,z')$ is uniformly $O(t^{\infty})$ for all $z$ and $z'$ in $\Omega$, and therefore
\[|E(t,z,z')|=O(t^{\infty}).\]


From here, the error may be iterated away using the usual Neumann series argument, as in \cite{mave} or section 4 of \cite{sherconicdegen}. Letting $*$ denote the operation of convolution in time and composition in space, define
\[K:=E-E*E+E*E*E-\dots.\]
It is an exercise in induction to see that $K(t,z,z')$ is well-defined and also $O(t^{\infty})$ as $t$ goes to zero, see for example the proof of parts a) and b) of Lemma 13 of \cite{sherconicdegen}. Note that $\Omega$ is compact, which is key. Then the difference of the true heat kernel and the patchwork heat kernel is
\[H^{\Omega}(t,z,z')-G(t,z,z')=-(G*K)(t,z,z').\]
As in Lemma 14 of \cite{sherconicdegen}, this can also be bounded in straightforward fashion by $O(t^{\infty})$, which completes the proof.
\qed 
\end{proof}

The key ingredients in this patchwork construction are:  (1) the model heat kernels satisfy off-diagonal decay estimates, and (2) the gradients of these model heat kernels satisfy similar estimates.  The argument can therefore be replicated in any situation where all models satisfy those estimates. Here is one generalization:
\begin{corollary} Using the notation of Theorem \ref{thml1}, suppose that $\Omega$ is compact and that the heat kernels on both $\Omega$ and $S$ satisfy off-diagonal bounds of the following form: if $A$ and $B$ are any two sets with $d(A,B)>0$, then uniformly for $z\in A$ and $z'\in B$, we have
\begin{equation}\label{eq:offdiagupperbounds}|H(t,z,z')|+|\nabla H(t,z,z')|=O(t^{\infty})\mbox{ as }t\to 0.\end{equation}
Then the conclusion of Theorem \ref{thml1} holds.
\end{corollary}
\begin{proof} Apply the same method, with a partition of unity on $\Omega$ consisting of just two components, one cutoff function for $\Omega_0$ where we use the model heat kernel $H^S$, and one cutoff function for the rest of $\Omega$ where we use $H^{\Omega}$. The result follows. \end{proof}

\begin{remark} The bounds \eqref{eq:offdiagupperbounds} are satisfied, for example, by Neumann heat kernels on compact, convex domains with no smoothness assumptions on the boundary \cite{wang1},  as well as by both Dirichlet and Neumann heat kernels on sectors, half-spaces, and Euclidean space.
\end{remark}

\subsection{Locality for Robin boundary condition}
In this section, we ask when locality results similar to those of Theorem \ref{thml1} hold for the Robin problem. The answer is that in many cases they may be deduced from locality of the Neumann heat kernels.   We consider a generalization of the classical Robin boundary condition \eqref{rbc1}
\begin{equation}\label{Schredinger_eq}
\frac{\partial}{\partial n}u(x)+c(x)u(x)=0, \qquad x\in \partial D.
\end{equation}

In the first version of the locality principle, to simplify the proof, we shall assume that $\Omega \subset S \subset \R^n$, and that $\Omega$ is bounded.  We note, however, that both of these assumption may be removed in the Corollary to the theorem.   The statement of the theorem may appear somewhat technical, so we explain the geometric interpretations of the assumptions.  Conditions (1) and (2) below are clear; they are required to apply our references \cite{wang1} and \cite{d}.  Items (3), (4), and (5) mean that the (possibly unbounded) domain, $D=S$ (and as in the corollary, in which $\Omega$ may be unbounded, $D=\Omega)$ has boundary which does not oscillate too wildly or ``bunch up'' and become space-filling at infinity.  These assumptions are immediately satisfied when the domains are bounded or if the boundary consists of finitely many straight pieces (like a sector in $\R^2$, for example). 

\begin{theorem}[Locality Principle for Robin Boundary Condition, v.1]  \label{loc_pr_for_R} 
Assume that $\Omega$ and $S$ are exact geometric matches on $\Omega_0$, as in Definition \ref{exactmatch}, with $\Omega _0 \subset \Omega \subset S \subset \mathbb R^n$. Assume that $\Omega$ is bounded. Let $K^S(t,x,y)$, $K^{\Omega}(t,x,y)$ be the heat kernels for the Robin Laplacian with boundary condition  \eqref{Schredinger_eq} for $D=S$ and $D=\Omega$, respectively, for the same $c(x) \in \cL^{\infty}(\partial S\cup \partial \Omega)$.  Let $\alpha:=\textrm{dist}(\Omega_0,S\setminus\Omega)$, and note that $\alpha>0$ by our assumption of an exact geometric match. Define the auxiliary domain
\[W:=\{x\in\Omega:d(x,\Omega_0)\leq \alpha/2\}.\]
We make the following (very mild) geometric assumptions:

\begin{enumerate}
\item Both $S$ and $\Omega$ satisfy the $(\eps,\delta)$-cone condition; 
\item Both $S$ and $\Omega$ have piecewise $\cC^3$ boundaries, and there exists a constant $\sigma \in \R$ such that the second fundamental form satisfies $\sff \geq - \sigma$ on both $\pa S$ and $\pa \Omega$.  
\item For any sufficiently small $r>0$ and any $t>0$, we have 
\begin{equation}\label{geom_of_S}
\sup_{x \in W}\int_{0}^{t}\int_{\partial S\setminus B(x,r)}\frac{1}{s^{\frac{n}{2}}}e^{-\frac{|x-z|^2}{s}}\sigma(dz)ds<\infty;
\end{equation}
\item For all $r>0$ and all $x\in\mathbb R^n$, and both $D=S$ and $D=\Omega$, there is a constant $C_D$ such that
\begin{equation}\label{geom_of_S2} \mathcal H^{n-1}(\partial D\cap (B(x,r))\leq C_D\textrm{Vol}_{n-1}(B_{n-1}(x,r)),
\end{equation}
where $\cH^{n-1}$ denotes the $n-1$ dimensional Hausdorff measure;
\item	If $G_n(x,y)$ is the free Green's function on $\mathbb R^n$, we have 
\begin{equation}\label{geom_of_Omega}
	\sup_{x\in W}\int_{\partial \Omega}G_n(x,y)\sigma(dy)<\infty.
\end{equation}
\end{enumerate}

Then, uniformly on $\overline{\Omega}_0\times \overline{\Omega}_0$, we have Robin locality:
	\begin{equation*}
	|K^S(t,x,y)-K^{\Omega}(t,x,y)|=O(t^{\infty}), \qquad t\rightarrow 0.
	\end{equation*}.  
\end{theorem}

The assumptions that $\Omega\subset S$ and that $\Omega$ is bounded can both be removed:

\begin{corollary}[Locality Principle for Robin Boundary Condition] Suppose we have an exact geometric match between $\Omega$ and $S$ on the bounded domain $\Omega_0$, and the Robin coefficient $c(x)$ agrees on a common open, bounded neighborhood $\Omega_c$ of $\Omega_0$ in $\Omega$ and $S$. Then, as long as Theorem \ref{loc_pr_for_R} holds for the pairs $(\Omega_0,\Omega)$ and $(\Omega_0,S)$, the conclusion of Theorem \ref{loc_pr_for_R} holds for the pair $(\Omega,S)$.
\end{corollary}

\begin{proof} Apply Theorem \ref{loc_pr_for_R} to the pairs $(\Omega_0,\Omega)$ and $(\Omega_0,S)$, using the same $W$, then use the triangle inequality. \qed
\end{proof}

Before we prove Theorem \ref{loc_pr_for_R}, we discuss the geometric assumptions \eqref{geom_of_S}, \eqref{geom_of_S2}, and \eqref{geom_of_Omega}, and give some sufficient conditions for them to hold.  First, observe that regardless of what $W$ is, \eqref{geom_of_S} is immediately valid if $S$ is a bounded domain whose boundary has finite one dimensional Lebesgue measure. It is also valid if $S$ is an infinite circular sector, by a direct computation, part of which is presented below.  

\begin{example}
	Let $S=S_{\gamma}\subset\mathbb{R}^2$ be a circular sector of opening angle $\gamma$ and infinite radius. Assume that $W$ and $\Omega$ are bounded domains such that $W\subset \Omega \subset S$,  and assume for simplicity that $W$ contains the corner of $S$, see figure \ref{fig2} (the case where this does not happen is similar.) Then \eqref{geom_of_S} holds. Indeed, let $r\in (0,\alpha/2)$ and $t>0$, then
	\begin{equation*}
	\begin{gathered}
	\sup_{x\in W}\int_{0}^{t}\int_{\partial S\setminus B(x,r)}\frac{1}{s}e^{-\frac{|x-z|^2}{s}}\sigma(dz)ds\\
	\leq
	\sup_{x\in W}\int_{0}^{t}\int_{\partial S\setminus \partial\Omega}\frac{1}{s}e^{-\frac{|x-z|^2}{s}}\sigma(dz)ds+\sup_{x\in W}\int_{0}^{t}\int_{(\partial S\cap \partial\Omega)\setminus B(x,r)}\frac{1}{s}e^{-\frac{|x-z|^2}{s}}\sigma(dz)ds\\
	\leq 2\int_{0}^{t}\int_{0}^{\infty}\frac{1}{s}e^{-\frac{\tau^2}{s}}d\tau ds+
	\int_{0}^{t}\frac{1}{s}e^{-\frac{r^2}{s}}\int_{(\partial S\cap \partial\Omega)}\sigma(dz)ds<\infty
	\end{gathered}
	\end{equation*}
	and
	\begin{equation*}
	\sup_{x\in W}\int_{\partial \Omega\cap B(x,r)}|\ln|x-z||\sigma(dz)
	\leq \int_{0}^{\alpha}|\ln \tau|d\tau<\infty.
	\end{equation*}
\end{example}

As for \eqref{geom_of_S2}, this is automatic if $D$ is a bounded domain with piecewise $\cC^1$ boundary. It is also true if $D$ is a circular sector (in fact here $C_D=2$).

The condition \eqref{geom_of_Omega} is also easy to satisfy:
\begin{proposition}\label{cond_for_geom_of_Omega}
	Assume that $\Omega$ is a bounded domain in $\R^n$ which has piecewise $\cC^3$ boundary. Let $W\subset \Omega$ be a compact set, then \eqref{geom_of_Omega} holds.
\end{proposition}

\begin{proof}Recall that
\begin{equation*}
	G_n(x,y)=
	\begin{cases}
	|\ln|x-y||, & \text{if } n=2;\\
	|x-y|^{2-n}, & \text{if } n\geq 3.\\
	\end{cases}
	\end{equation*}
	Since $W$ is compact, it is enough to prove that
	\begin{equation}\label{geom_of_Omega1}
	x\mapsto \int_{\partial \Omega}G_n(x,y)\sigma(dy)
	\end{equation}
	is a continuous function on $W$.
	
	Fix $x\in W$. Let $\varepsilon>0$ and $\{x_j\}_{j=1}^{\infty}\subset W$ be a sequence such that $x_j\rightarrow x$. Since $\partial \Omega$ is piecewise $\cC^3$, and $G_n(x,y)$ is in $\cL^1 _{loc}$, we can choose $\delta>0$ such that
	\begin{equation}\label{geom_of_Omega2}
	\int_{\partial \Omega\cap B(x, 2\delta)}G_n(x,y)\sigma(dy)<\varepsilon,
	\qquad
	\int_{\partial \Omega\cap B(x, 2 \delta)}G_n(x_j,y)\sigma(dy)<\varepsilon,
	\end{equation}
	for sufficiently large $j\in \mathbb{N}$, such that for these $j$ we also have $|x-x_j| < \delta$.  
	To see this, we note that $G_n (x, y) = G_n (|x-y|) = G_n(r)$, where $r = |x-y|$, and similarly, $G_n (x_j, y) = G_n (r_j)$ with $r_j = |x_j - y|$.  Thus, choosing the radius, $2 \delta$, sufficiently small, since $G_n$ is locally $\cL^1(\partial\Omega)$ integrable, and $\pa \Omega$ is piecewise $\cC^3$, we can make the above integrals as small as we like.  
	
	Now, we note that $G_n(x_j,y)\rightarrow G_n(x,y)$ as $j\rightarrow\infty$, for $y\in \partial \Omega\setminus B(x, 2 \delta)$.  Moreover, since $\Omega$ and thus $\pa \Omega$ are both compact, $G_n(x_j,y)<C=C(\delta)$ for $y \in \partial \Omega \setminus B(x, 2\delta)$.  The Dominated Convergence Theorem therefore implies
	\begin{equation*}
	\left|\int_{\partial \Omega\setminus B(x, 2 \delta)}G_n(x,y)-G_n(x_j,y)\sigma(dy)\right|<\varepsilon
	\end{equation*}
	for sufficiently large $j\in \mathbb{N}$. This, together with \eqref{geom_of_Omega2}, implies that the function \eqref{geom_of_Omega1} is continuous on $W$.
\end{proof}

In summary, we have 
\begin{corollary} The locality principle, Theorem \ref{loc_pr_for_R}, holds in the case where $\Omega$ is a bounded domain in $\mathbb R^n$ with piecewise $\cC^3$ boundary, and $S$ is any domain with piecewise $\cC^3$ boundary such that $\Omega$ and $S$ are an exact geometric match on the bounded subdomain $\Omega_0$ as in Definition \ref{exactmatch}.  Moreover, we assume that:  
\begin{enumerate} 
\item Both $S$ and $\Omega$ satisfy the $(\eps,\delta)$-cone condition; 
\item There exists a constant $\sigma \in \R$ such that the second fundamental form satisfies $\sff \geq - \sigma$ on both $\pa S$ and $\pa \Omega$; 
\item $S$ satisfies \eqref{geom_of_S} and \eqref{geom_of_S2}; 
\item The Robin coefficient $c(x)$ agrees on a common open bounded neighborhood $\Omega_c$ in $\Omega$ and $S$.  
\end{enumerate} 
\end{corollary}

\begin{remark} In particular, all assumptions are satisfied if $\Omega$ is a bounded polygonal domain in $\mathbb R^2$ and $S$ is a circular sector in $\mathbb R^2$. \end{remark}

Now we prove Theorem \ref{loc_pr_for_R}.

\begin{proof} 
Since the domains $S$ and $\Omega$ satisfy the interior cone condition, there are Gaussian upper bounds for the corresponding Neumann heat kernels as given in \eqref{est_for_PC_Pomega}, 
With this in mind, define
\begin{equation*}
\begin{gathered}
F_1(t):=\sup_{(s,x,z)\in (0,t]\times W\times (\overline{S\cap\Omega})}\left|H^S(s,x,z)-H^{\Omega}(s,x,z)\right|,\\
F_2(t):=\sup_{(s,x,z)\in (0,t]\times W\times S\setminus\Omega}\left|H^S(s,x,z)\right|,\\ 
F_3(t):=\sup_{(s,x,z)\in (0,t]\times W\times \partial\Omega\setminus \partial S}\left|H^{\Omega}(s,x,z)\right|.
\end{gathered}
\end{equation*}
It now follows from \eqref{est_for_PC_Pomega} and Theorem \ref{thml1} that
\begin{equation}\label{asympt_F}
F(t):=\max(F_1(t), F_2(t), F_3(t))=O(t^{\infty}),
\qquad
t\rightarrow 0.
\end{equation}

The reason we require the Neumann heat kernels is because, as in \cite{P,zayed}\footnote{We note that the result is stated for compact domains. However, the construction is purely formal and works as long as the series converges. Under our assumptions, we shall prove that it does.}, the Robin heat kernels, $K^S(t,x,y)$ and $K^{\Omega}(t,x,y)$, can be expressed in terms of $H^S(t,x,y)$ and $H^{\Omega}(t,x,y)$ in the following way.  Define 
\begin{equation*}
k_0^D(t,x,y)=H^D(t,x,y)
\end{equation*}
and
\begin{equation*}
k_m^{D}(t,x,y)=\int_{0}^{t}\int_{\partial D} H^D(s,x,z)c(z)k_{m-1}^D(t-s,z,y)\sigma(dz)ds
\end{equation*}
for $m\in \mathbb{N}$.  Then
\begin{equation*}
K^S(t,x,y)=\sum_{m=1}^{\infty}k_m^{S}(t,x,y),
\qquad
K^{\Omega}(t,x,y)=\sum_{m=1}^{\infty}k_m^{\Omega}(t,x,y).
\end{equation*}

Let us define the function
	\begin{equation} \label{defA} 
	\begin{gathered}
	A(t,x):=\int_{0}^{t}\int_{\partial S}\left|H^S(s,x,z)c(z)\right|\sigma(dz)ds+\int_{0}^{t}\int_{\partial\Omega}\left|H^{\Omega}(s,x,z)c(z)\right|\sigma(dz)ds\\
	=:A_1(t,x)+A_2(t,x).
	\end{gathered}
	\end{equation}
	on $(0,1]\times W$. The following lemma, in particular, shows that $A(t,x)$ is a well defined function.

\begin{lemma}\label{lemma_A}
		The function $A(t,x)$ is uniformly bounded on $(0, 1]\times W$.
	\end{lemma}
	\begin{proof}
		For $n=1$ the lemma follows from \eqref{est_for_PC_Pomega}. Hence, we assume here $n\geq 2$. For any $x\in\overline W$, $A_j(t,x)$, $j=1,2$, is an increasing function with respect to the variable $t\in (0,1]$. Therefore, it is sufficient to prove that $A_j(x):=A_j(1,x)$ is bounded on $W$, for $j=1,2$. 
%
		
		Let us choose $0<\rho<\min(\alpha/2,1)$. Without loss of generality, setting $C_1=C_2=1$ in \eqref{est_for_PC_Pomega}, we obtain
		\begin{equation*}
		\begin{gathered}
		A_1(x)+A_2(x)\\
		\leq\int_{0}^{1}\int_{\partial S\setminus B(x,\rho)}s^{-\frac{n}{2}}e^{-\frac{|x-z|^2}{s}}|c(z)|\sigma(dz)ds+
		\int_{0}^{1}\int_{\partial S\cap B(x,\rho)}s^{-\frac{n}{2}}e^{-\frac{|x-z|^2}{s}}|c(z)|(dz)ds\\
		+ \int_{0}^{1}\int_{\partial \Omega\setminus B(x,\rho)}s^{-\frac{n}{2}}e^{-\frac{|x-z|^2}{s}}|c(z)|\sigma(dz)ds+
		\int_{0}^{1}\int_{\partial \Omega\cap B(x,\rho)}s^{-\frac{n}{2}}e^{-\frac{|x-z|^2}{s}}|c(z)|(dz)ds\\
		=:J_1(x)+J_2(x)+J_3(x)+J_4(x).
		\end{gathered}
		\end{equation*}
		The boundedness of $J_1(x)$ on $W$ follows from \eqref{geom_of_S}. For $J_3(x)$ we estimate using only that $\partial\Omega$ is bounded and thus, since it is piecewise $\cC^3$, has finite measure,
		\begin{equation*}
		J_3(x)\leq \|c\|_{\infty}\int_{0}^{1}\frac{1}{s^{\frac{n}{2}}}e^{-\frac{\rho^2}{s}}\int_{\partial \Omega\setminus B(x,\rho)}\sigma(dz)<\infty.
		\end{equation*}		
		Finally, since $\rho<\alpha/2$, $\partial S\cap B(x,\rho)=\partial \Omega\cap B(x,\rho)$ for $x\in W$, and hence by Fubini's theorem and a change of variables
		\begin{equation*}
		\begin{gathered}
		J_2(x)=J_4(x)=\int_{0}^{1}\int_{\partial \Omega\cap B(x,\rho)}s^{-\frac{n}{2}}e^{-\frac{|x-z|^2}{s}}|c(z)|(dz)ds\\
		\leq \|c\|_{\infty}\int_{\partial\Omega\cap B(x,\rho)}\frac{1}{|x-z|^{n-2}}\int_{|x-z|^2}^{+\infty}\tau^{\frac{n}{2}-2}e^{-\tau}d\tau \sigma(dz).
		\end{gathered}
		\end{equation*}
		For $n>2$, the second integrand is bounded, and hence, Proposition \ref{geom_of_Omega} implies that $J_2(x)$ and $J_4(x)$ are bounded on $W$. If on the other hand $n=2$, then
		\begin{equation*}
		J_2(x)=J_4(x)=\int_{\partial \Omega\cap B(x,\rho)}|c(z)|\int_{|x-z|^2}^{+\infty}\tau^{-1}e^{-\tau}d\tau \sigma(dz).
		\end{equation*}
		Since $\rho < 1$, $\rho^2<\rho<1$, so we can write
		\begin{equation*}
		\begin{gathered}
		J_2(x)=J_4(x)\leq\int_{\partial \Omega\cap B(x,\rho)}|c(z)|\int_{|x-z|^2}^{1}\tau^{-1}d\tau \sigma(dz)+\int_{\partial \Omega\cap B(x,\rho)}|c(z)|\int_{1}^{+\infty}e^{-\tau}d\tau \sigma(dz)\\
		\leq \|c\|_{\infty}\int_{\partial \Omega\cap B(x,\rho)}\left|\ln|x-z|^2\right|\sigma(dz)+\|c\|_{\infty}\int_{\partial \Omega\cap B(x,\rho)} \sigma(dz),
		\end{gathered}
		\end{equation*}
		which is finite by Proposition \ref{geom_of_Omega}, the boundedness of $\partial \Omega$ and the piecewise $\cC^3$ smoothness of the boundary.
\qed 
	\end{proof}

\begin{corollary}\label{cor:notjustboundedbutshrinking} In the notation of Lemma \ref{lemma_A}, we have
\[\lim_{T\to 0}\sup_{(t,x) \in(0,T] \times W}A(t,x)=0.\]
\end{corollary}
\begin{proof} Consider the functions $A(t,x)$. They are monotone increasing in $t$ for each $x$, and they are continuous in $x$ for each $t$ by continuity of solutions to the heat equation. We claim that as $t\to 0$, $A(t,x)$ approaches zero pointwise. To see this write the time integral from $0$ to $t$ in each $A_j(t,x)$, $j=1,2$, as a time integral over $[0,1]$ by multiplying the integrand by the characteristic function $\chi_{[0,t]}$. For example,
\[A_1(t,x)=\int_0^1\int_{\partial S} \chi_{[0,t]}|H^S(s,x,z)c(z)|\, \sigma(dz)ds.\]
The integrands are bounded by $|H^S(s,x,z)c(z)|$, which is integrable by Lemma \ref{lemma_A}, and for each $x$, they converge to zero as $t \to 0$. So by the Dominated Convergence Theorem applied to each $A_j(t,x)$, we see that $A(t,x)\to 0$ as $t \to 0$ for each $x$.

Now we have a monotone family of continuous functions converging pointwise to a continuous function (zero) on the compact set $W$. By Dini's theorem, this convergence is in fact uniform, which is precisely what we want.
\qed
\end{proof}

To use this, fix a small number $A$ to be chosen later. Then Corollary \ref{cor:notjustboundedbutshrinking} allows us to find $T>0$ such that
	\begin{equation}\label{eq:defofA}
	A(t,x)<A,
	\qquad
	(t,x)\in (0,T]\times W.
	\end{equation}
	Next we prove the following two auxiliary propositions.
	\begin{proposition}\label{prop:estforkn}
		The following inequality holds with $D=S$ and $D=\Omega$:
		\begin{equation}\label{estimate_for_kn}
		\begin{gathered}
		\int_{0}^{t}\int_{\partial D}\left|k_m^D(s,x,z)c(z)\right|\sigma(dz)ds\leq 2^{m+1}A^{m+1}\\
		\end{gathered}
		\end{equation}
		on $(0,T]\times W$, for any $m\in \mathbb{N}$. Moreover, an identical inequality holds when $k_m^D(s,x,z)$ is replaced by $k_m^D(s,z,x)$.
	\end{proposition}
	
	\begin{proof}
		By induction. For $m=0$, recalling the definition of $A$, \eqref{defA}, 
		$$ \int_0 ^t \int_{\partial D} |k_0 ^D (s,x,z) c(z)| \sigma(dz) ds = \int_0 ^t \int_{\partial D} |H^D (s,x,z) c(z) |\sigma(dz) ds \leq A(t,x) < A.		$$
		We have thus verified the base case.  Now, we assume that \eqref{estimate_for_kn} holds for $k\leq m$. Consider $k=m+1$:
		
		\begin{equation*}
		\begin{gathered}
		\int_{0}^{t}\int_{\partial D}\left|k_{m+1}^D(s,x,z)c(z)\right|\sigma(dz)ds\\
=\int_{0}^{t}\int_{\partial D}\int_{0}^{s}\int_{\partial D}\left|H^D(\tau,z,\zeta)k_{m}^D(s-\tau,\zeta,x)c(\zeta)c(z)\right|\sigma(d\zeta) d\tau \sigma(dz) ds.
		\end{gathered}
		\end{equation*}
		Changing variables:
		\begin{equation}\label{estimate_for_kn1}
		\begin{gathered}
		\int_{0}^{t}\int_{\partial D}\int_{0}^{s}\int_{\partial D}\left|H^D(\tau,z,\zeta)k_{m}^D(s-\tau,\zeta,x)c(\zeta)c(z)\right|\sigma(d\zeta) d\tau \sigma(dz)ds\\
		\leq\int_{0}^{t}\int_{\partial D}\int_{0}^{s}\int_{\partial D}\left|H^D(s-\tau,z,\zeta)k_{m}^D(\tau,\zeta,x)c(\zeta)c(z)\right|\sigma(d\zeta) d\tau \sigma(dz)ds\\
		\leq\int_{0}^{t}\int_{\partial D}\int_{0}^{t}\int_{\partial D}\left|H^D(|s-\tau|,z,\zeta)k_{m}^D(\tau,\zeta,x)c(\zeta)c(z)\right|\sigma(d\zeta) d\tau \sigma(dz)ds\\
		\leq\int_{0}^{t}\int_{\partial D}\left(\int_{0}^{t}\int_{\partial D}\left|H^D(|s-\tau|,z,\zeta)c(z)\right|\sigma(dz)ds\right)\left|k_{m}^D(\tau,\zeta,x)c(\zeta)\right|\sigma(d\zeta) d\tau .
		\end{gathered}
		\end{equation}
		For the integrand, we compute
		\begin{equation*}
		\begin{gathered}
		\int_{0}^{t}\int_{\partial D}\left|H^D(|s-\tau|,z,\zeta)c(z)\right|\sigma(dz)ds\\
		=\int_{0}^{\tau}\int_{\partial D}\left|H^D(|s-\tau|,z,\zeta)c(z)\right|\sigma(dz)ds+\int_{\tau}^{t}\int_{\partial D}\left|H^D(|s-\tau|,z,\zeta)c(z)\right|\sigma(dz)ds\\
		=\int_{0}^{\tau}\int_{\partial D}\left|H^D(\tau-s,z,\zeta)c(z)\right|\sigma(dz)ds+\int_{0}^{t-\tau}\int_{\partial D}\left|H^D(s,z,\zeta)c(z)\right|\sigma(dz)ds<2A.
		\end{gathered}
		\end{equation*}
		Therefore, from the induction hypothesis and \eqref{estimate_for_kn1}, we obtain
		\begin{equation*}
		\begin{gathered}
		\int_{0}^{t}\int_{\partial D}\int_{0}^{s}\int_{\partial D}\left|H^D(\tau,z,\zeta)k_{m}^D(s-\tau,\zeta,x)c(\zeta)c(z)\right|\sigma(d\zeta) d\tau \sigma(dz)ds\\
		\leq 2A\int_{0}^{t}\int_{\partial D}\left|k_{m}^D(\tau,\zeta,x)c(\zeta)\right|\sigma(d\zeta) d\tau \leq 2A\cdot 2^{m+1}A^{m+1}=2^{m+2}A^{m+2},
		\end{gathered}
		\end{equation*}
as desired. 

The estimates with $x$ and $z$ reversed are proved similarly. Note in particular that the base case works because $k_0^D=H_0^D$ is a Neumann heat kernel and is thus symmetric in its spatial arguments.
\qed 
	\end{proof}

We need one more lemma concerning pointwise bounds for $k_m^D$, which uses the geometric assumption \eqref{geom_of_S2}:
\begin{lemma}\label{lem:offdiagfix} Let $D=S$ or $\Omega$. There exists $T_0>0$ such that for all $m$, all $t<T_0$, all $x\in D$, and all $y\in D$,
\[|k_m^D(t,x,y)|\leq \frac{C_1}{2^m} t^{- \frac n 2} e^{-\frac{|x-y|^2}{C_2 t}}.\]
\end{lemma}
\begin{proof} The proof proceeds by induction. The base case is $m=0$, which is \eqref{est_for_PC_Pomega}.

Now assume we have the result for $k=m$. Using the iterative formula, we have
\begin{equation}|k_{m+1}^D(t,x,y)|\leq ||c||_{\infty}\int_0^t\int_{\partial D}|H^D(s,x,z)k^D_m(t-s,z,y)|\, \sigma(dz)ds.\end{equation}
Using \eqref{est_for_PC_Pomega} and the inductive hypothesis, we see that the integrand is bounded by
\[C_1C_1 2^{-m} s^{-n/2}(t-s)^{-n/2} e^{-\frac{1}{C_2}(\frac{|x-z|^2}{s}+\frac{|z-y|^2}{t-s})}.\]

First assume that $D$ is a half-space. We do the estimate in the case $n=2$, because the general case is analogous.  Hence, we use the coordinates $x = (x_1, x_2)$, $y=(y_1, y_2)$, $z=(z_1, z_2)$, and estimate using $\{z_2 = 0 \} \subset \R^2$ for $\pa D$.  Dropping the constant factors, and saving the integral with respect to time for later, we therefore estimate 
$$\int_\R s^{-1} (t-s)^{-1} e^{- \frac{|x-z|^2}{C_2s} - \frac{|y-z|^2}{C_2(t-s)}} dz_1.$$
Without loss of generality, we shall assume that $x=(0,0)$.  Then we are estimating
$$\int_\R s^{-1} (t-s)^{-1} e^{ \frac{ -z_1 ^2 (t-s) - s |y-z|^2}{C_2 s(t-s)}} dz_1.$$
Since $z \in \pa D$, we have $z_2=0$. For the sake of simplicity, set $y_2=0$; the case where $y_2$ is nonzero is similar.  Given this assumption, we set
$$z := z_1, \quad y:=y_1,$$
and estimate 
$$\int_\R s^{-1} (t-s)^{-1} e^{ \frac{-z^2 (t-s) - s (y-z)^2}{C_2 s(t-s)}} dz.$$
We do the standard trick of completing the square in the exponent.  This gives 
$$\int_\R s^{-1} (t-s)^{-1} \exp \left[ - \left( \frac{ \sqrt{t} z - \frac{s y}{\sqrt{t}} }{ \sqrt{C_2} \sqrt{s} \sqrt{t-s}} \right)^2 - \frac{y^2}{C_2(t-s)} + \frac{sy^2}{C_2t(t-s)} \right] dz.$$
We therefore compute the integral over $\R$ in the standard way, obtaining 
$$s^{-1/2} (t-s)^{-1/2} \sqrt {\frac{C_2\pi}{t}} e^{- \frac{y^2}{C_2(t-s)} + \frac{sy^2}{C_2t(t-s)}} = s^{-1/2} (t-s)^{-1/2} \sqrt {\frac{C_2\pi}{t}} e^{ \frac{-ty^2 + sy^2}{C_2t(t-s)}}$$
$$=s^{-1/2} (t-s)^{-1/2} \sqrt {\frac{C_2\pi}{t}} e^{-\frac{y^2}{C_2t}}.$$
Finally, we compute the integral with respect to $s$, 
$$\int_0 ^t \frac{1}{\sqrt{s}} \frac{1}{\sqrt{t-s}} ds = \pi.$$
Hence, the total expression is bounded from above by 
$$\pi \sqrt {\frac{C_2 \pi}{t}} e^{-\frac{y^2}{C_2t}}.$$
Since we had assumed that $x=0$, we see that this is indeed 
$$\pi \sqrt{\frac{C_2 \pi}{t}} e^{-\frac{|x-y|^2}{C_2t}}.$$
Recalling the constant factors, we have
$$|k_{m+1}^D(t,x,y)|\leq C_1 C_1 ||c||_{\infty} 2^{-m} \pi \sqrt{\frac{C_2 \pi}{t}} e^{-\frac{|x-y|^2}{C_2t}}.$$
Now we note that the power of $t$ is $t^{-(n-1)/2}$ for dimension $n=2$.  Hence, we re-write the above estimate as 
$$|k_{m+1}^D(t,x,y)|\leq C_1 C_1 ||c||_{\infty} 2^{-m} \pi \sqrt{t} t^{-1} \sqrt{C_2 \pi} e^{-\frac{|x-y|^2}{C_2t}}.$$
We then may choose for example 
$$t\leq T_0 = \frac{ 1}  {4(C_1+1)^2 (||c||_\infty+1)^2  \pi^3 (C_2+1)}$$
$$ \implies \sqrt{t} \leq \frac{1}{2 (C_1+1) (||c||_\infty +1) \pi^{\frac 3 2} \sqrt{C_2 +1}}.$$ 
This ensures that 
$$|k_{m+1}^D(t,x,y)|\leq C_1 2^{-(m+1)} t^{-\frac n 2} e^{-\frac{|x-y|^2}{C_2t}}, \quad n=2.$$
We note that in general, for $\R^n$, by estimating analogously, noting that the integral will be over $\R^{n-1}$, we obtain 
$$|k_{m+1}^D(t,x,y)|\leq ||c||_{\infty} 2^{-m} \pi (C_2 \pi)^{n/2} t^{-(n-1)/2} e^{-\frac{|x-y|^2}{C_2t}}.$$
So, in the general-$n$ case, we let 
$$T_0 =  \frac{ 1}  {4(C_1+1)^2 (||c||_\infty+1)^2  \pi^{2+n} (C_2+1)^n}.$$ 
Then, for all $t \leq T_0$, we have 
$$|k_{m+1}^D(t,x,y)|\leq C_1 2^{-m-1} t^{-\frac n 2} e^{-\frac{|x-y|^2}{C_2t}}.$$

Now consider the case where $D$ is a general domain, not necessarily a half-space.  As before, we have
\begin{equation}\label{eq:aux1}|k_{m+1}^D(t,x,y)|\leq \frac{||c||_{\infty}C_1^2}{2^m}\int_0^t\int_{\partial D}s^{-n/2}(t-s)^{-n/2} e^{-\frac{1}{C_2}(\frac{|x-z|^2}{s}+\frac{|z-y|^2}{t-s})}\, \sigma(dz)ds.\end{equation}
We claim that the right-hand side of \eqref{eq:aux1} is less than or equal to $C_D$, the constant from \eqref{geom_of_S2}, times the corresponding integral in the case where $\partial D$ is a half-plane through $x$ and $y$. Assuming this claim, we get the same bound as for a half-plane, but with an extra $C_D$, and adjusting $T_0$ to absorb $C_D$ as well, by putting an extra $(C_D+1)^2$ in the denominator, completes the proof.

To prove this claim, we use the so-called layer cake representation: rewrite the right-hand side of \eqref{eq:aux1}, without the outside constants, as
\begin{equation}\label{layercake}\int_0^ts^{-n/2}(t-s)^{-n/2}\int_{\partial D}\int_0^{\infty}\chi_{\{f(s,t,x,y,z)<a\}}e^{-a}\, da\sigma(dz)ds,\end{equation}
where naturally
\[f(s,t,x,y,z):=\frac{1}{C_2}\left(\frac{|x-z|^2}{s}+\frac{|z-y|^2}{t-s}\right).\]
The representation \eqref{layercake} may seem odd at first but reverts to \eqref{eq:aux1} upon integration in $a$. Switching the order of integration in \eqref{layercake} (valid by Fubini-Tonelli , since everything is positive) and evaluating the $z$-integral, this becomes
\begin{equation}\label{layercake2}\int_0^ts^{-n/2}(t-s)^{-n/2}\int_0^{\infty}\cH^{n-1}(\partial D\cap \{z:f(s,t,x,y,z)<a\})e^{-a}\, dads. \end{equation}
Let us more closely examine the set $\{z:f(s,t,x,y,z)<a\}$. It is the set where
\[\left(1-\frac st \right)|x-z|^2+\frac st|z-y|^2<\frac 1tC_2 a s (t-s),\]
and hence it is a ball centered at the point $P(s,t,x,y):=(1-\frac st)x+\frac st y$, with radius squared equal to
\[R^2(s,t,x,y):=\frac 1tC_2 a s (t-s)-\frac st(1-\frac st)|y-x|^2,\]
or to zero if the right-hand side is not positive. Therefore \eqref{layercake2} equals
\begin{equation}\label{layercake3}\int_0^ts^{-n/2}(t-s)^{-n/2}\int_0^{\infty}\cH^{n-1}(\partial D\cap B_n(P,R))e^{-a}\, dads. \end{equation}
By the assumption \eqref{geom_of_S2}, this is bounded by
\begin{equation}\label{layercake3a}C_D\int_0^ts^{-n/2}(t-s)^{-n/2}\int_0^{\infty}\textrm{Vol}_{n-1}(B_{n-1}(P,R))e^{-a}\, dads. \end{equation}

However, in the event that $D$ is a half-space with $x$ and $y\in\partial D$ (so also $P\in\partial D$), we have $\partial D\cap B_n(P,R)=B_{n-1}(P,R)$, so \eqref{layercake3} equals
\begin{equation}\label{layercake4}\int_0^ts^{-n/2}(t-s)^{-n/2}\int_0^{\infty}\textrm{Vol}_{n-1}(B_{n-1}(P,R))e^{-a}\, dads. \end{equation}
Therefore, the integral \eqref{layercake3} for general $D$ is bounded by $C_D$ times the integral \eqref{layercake3} for a half-space. Since \eqref{layercake3} is equal to the right-hand side of \eqref{eq:aux1} without the preceding constants, the claim is proven. This completes the proof of Lemma \ref{lem:offdiagfix}.  
\end{proof}

\begin{remark}
The key is that the integral is half an order better in $t$ than the true heat kernel, which is a critical feature of the difference between Robin and Neumann heat kernels. It allows us to utilize the extra $\sqrt t$ to obtain the additional factor of $2^{-m}$ which is required for the induction step in the next proposition.
\end{remark}

Now, we establish the main estimate to prove Theorem \ref{loc_pr_for_R}. Let
\[G(t)=\max \left\{F(t), \, 2C_1t^{-(n/2)}e^{-\frac{(\alpha/2)^2}{C_2 t}}\right\}.\]
We note that of course we still have $G(t)=O(t^{\infty})$.

	\begin{proposition}\label{p: estimate_for_kn_min_kn}
		There exists $T>0$ such that the estimate
		\begin{equation} \label{estimate_for_kn_min_kn}
		|k_m^S(t,x,y)-k_m^{\Omega}(t,x,y)|\leq G(t)\cdot 7\cdot 2^{-m}
		\end{equation}
		holds for all $(t,x,y)\in (0,T]\times W\times\overline{\Omega}_0$. 
	\end{proposition}
	
	\begin{proof} We choose $T$ small enough so that $T<T_0$ in Proposition \ref{p: estimate_for_kn_min_kn} and so that \eqref{eq:defofA} holds with $A=1/4$. 

Now proceed by induction. The base case is instantaneous by definition of $k_0$ and of $F(t)$, using the known Neumann locality. So assume that \eqref{estimate_for_kn_min_kn} holds for $k=m$; we will prove it for $k=m+1$. Using some algebraic manipulations,
		\begin{equation*}
		\begin{gathered}
		I:=|k_{m+1} ^S(t,x,y)-k_{m+1}^{\Omega}(t,x,y)|\leq I_1+I_2+I_3\\
		:=
		+\int_{0}^{t}\int_{\partial S\cap\partial\Omega}\left|H^S(s,x,z)k_m^S(t-s,z,y)-H^{\Omega}(s,x,z)k_m^{\Omega}(t-s,z,y)\right||c(z)|\sigma(dz)ds\\
		+ \int_{0}^{t}\int_{\partial S\setminus\partial\Omega}|H^S(s,x,z)k_m^S(t-s,z,y)c(z)|\sigma(dz)ds\\
		+ \int_{0}^{t}\int_{\partial\Omega\setminus\partial S}|H^{\Omega}(s,x,z)k_m^{\Omega}(t-s,z,y)c(z)|\sigma(dz)ds\\
		\end{gathered}
		\end{equation*}
		We estimate these terms separately, beginning with $I_1$.
		\begin{equation*}
		\begin{gathered}
		I_1\leq \int_{0}^{t}\int_{\partial S\cap\partial\Omega}\left|H^S(s,x,z)-H^{\Omega}(s,x,z)\right||k_m^S(t-s,z,y)||c(z)|\sigma(dz)ds\\
		+\int_{0}^{t}\int_{W\cap \partial S\cap\partial\Omega }\left| k_m^S(t-s,z,y)-k_m^{\Omega}(t-s,z,y)\right|\left|H^{\Omega}(s,x,z)\right||c(z)|\sigma(dz)ds\\
		+\int_{0}^{t}\int_{(\Omega\setminus W)\cap \partial S\cap\partial\Omega}\left| k_m^S(t-s,z,y)-k_m^{\Omega}(t-s,z,y)\right|\left|H^{\Omega}(s,x,z)\right||c(z)|\sigma(dz)ds.	\end{gathered}
		\end{equation*}
The first term in the first integral is bounded by $F(t)$, since $x\in W$ and $z\in\overline{\Omega}$, so we may pull it out. We estimate the other term with Proposition \ref{prop:estforkn} and get a bound of $F(t)\cdot 2^{m+1}A^{m+1}=F(t)\cdot 2^{-(m+1)}$ for the first integral.

For the second integral, we pull out the supremum of the first term using the inductive hypothesis. We estimate the other term using the definition of $A$ and we get a bound of $G(t)\cdot 7\cdot 2^{-m-2}$.

For the third integral, we use Lemma \ref{lem:offdiagfix} to pull out the first term, ignoring the difference and just estimating both $k$ terms separately. Since $|z-y|\geq\alpha/2$ on this region, the supremum is less than $2^{-m}G(t)$ by Proposition \ref{p: estimate_for_kn_min_kn}. We estimate the other term using the definition of $A$ and we get $1/4$, giving a bound of $2^{-m-2}G(t)$. Overall, we have
\[I_1\leq G(t)(2^{-m-1}+7\cdot 2^{-m-2}+2^{-m-2}).\]

		Next we estimate the terms $I_2$, $I_3$. In each, we pull out the supremum of $H^S(s,x,z)$ over the relevant region, and observe that it is bounded above by $F(t)$. For the term remaining in the integral we use Proposition \ref{prop:estforkn}. Since $F(t)\leq G(t)$, we obtain a bound of $G(t)\cdot 2^{-m-1}$ for each of these two terms. Putting it all together, we see
\[I\leq G(t)(3\cdot 2^{-m-1}+7\cdot 2^{-m-2}+2^{-m-2})=G(t)\cdot 2^{-m-1}\left(3+\frac 72+\frac 12\right)=G(t)\cdot 7\cdot 2^{-m-1},\]
as desired.
	\qed 
	\end{proof}

Finally, we prove Theorem \ref{loc_pr_for_R}.  By Proposition \ref{estimate_for_kn_min_kn},
	\begin{multline*}
	|K^S(t,x,y)-K^{\Omega}(t,x,y)|\leq \sum_{m=0}^{\infty}|k _ m ^S(t,x,y)-k _ m ^{\Omega}(t,x,y)|\\
\leq \sum_{m=0}^{\infty}7G(t)2^{-m}=14\cdot G(t),
	\end{multline*}
	which is $O(t^{\infty})$ as $t\rightarrow 0$. \qed 
\end{proof}

\section{Hearing the corners of a drum}  \label{hearing} 
As a consequence of the work in the previous section, the locality principle holds for Robin conditions $\Omega$ is a bounded polygonal domain in $\mathbb R^2$, and $S$ is a whole space, half-space, or circular sector respectively.  Therefore, to compute the heat trace expansion for a polygonal domain $\Omega \subset \R^2$, it suffices to chop the domain into pieces and, depending on the piece, replace the true  heat kernel with one of the following:
\begin{itemize}
\item the heat kernel for $\R^2$ away from the boundary of $\Omega$, 
\item the heat kernel for a half plane (with the same boundary condition taken on the boundary as taken on $\pa \Omega$) near $\pa \Omega$ but away from the corners, 
\item the heat kernel for an infinite circular sector with opening angle equal to that at a corner (with the same boundary condition taken on the boundary as taken on $\pa \Omega$), near a corner of $\Omega$.
\end{itemize}

Henceforth we use the classical Robin boundary condition as in \eqref{rbc1}, so that in \eqref{Schredinger_eq}, $c(x)$ is a constant.

The heat kernel for $\R^2$ is given by \eqref{hkrn} with $n=2$, 
$$H_{\R ^2} (t,z, z') = \frac{1}{4\pi t}e^{-\frac{|z-z'|^2}{4t}}.$$
The heat kernels for the half plane, 
$$\R^2 _+ := \{ (x,y) \in \R^2 \mid y \geq 0 \}$$
with the Dirichlet and Neumann boundary conditions at $\{y=0\}$ are given by the method of images.  For the Neumann condition, 
$$H_{\mathbb R_{+}^2,\textrm{Neumann}} = H_{\mathbb R^2}(t,x,y,x',y') + H_{\mathbb R^2}(t,x,y,x',-y'), \quad z=(x,y), \quad z'= (x', y'),
$$
whereas for the Dirichlet condition, 
$$H_{\mathbb R_{+}^2,\textrm{Dirichlet}} = H_{\mathbb R^2}(t,x,y,x',y') - H_{\mathbb R^2}(t,x,y,x',-y'), \quad z=(x,y), \quad z'= (x', y').
$$
The Robin heat kernel, 
\[H_{\mathbb R_+^2,\textrm{Robin}}:=H_{\mathbb R_{+}^2,\textrm{Neumann}}+H_{corr},\]
where
$$H_{corr}(t,x,y,x',y'):= - \frac{\alpha}{\beta \sqrt{4\pi t}} e^{-\frac{(x-x')^2}{4t}} e^{\frac{ \alpha(y+y')}{\beta}} e^{\frac{\alpha^2 t}{\beta^2}} \erfc \left( \frac{y+y'}{\sqrt{4t}} + \frac \alpha \beta \sqrt t \right).$$
Above, $\alpha$ and $\beta$ are given in \eqref{rbc1}.  Recall that the complementary error function is smooth in $z$, bounded by $1$ for $z\geq 0$, and decaying to infinite order as $z\to\infty$:
$$\erfc(z) = 1 - \erf(z) = \frac{2}{\sqrt \pi} \int_z ^\infty e^{-s^2}ds.$$
The heat kernels for an infinite sector are more complicated. Nevertheless, they can be computed via the method of Green's functions.

In forthcoming work, we compute the Green's function for a circular sector of opening angle $\gamma$ for the Laplacian with Dirichlet, Neumann, and Robin boundary conditions.  Here, we simply present the result of the calculations.  For the \textbf{Robin} boundary condition, 
\begin{equation} \label{RobinGK} 
  G_R(s,r,\phi,r_0,\phi_0)=\frac{1}{\pi^2}\int_{0}^{\infty}K_{i\mu}(r \sqrt s)K_{i\mu}(r_0 \sqrt s)
\end{equation}

\begin{equation*}
  \times\biggl\{ \cosh(\pi-|\phi_0-\phi|)\mu+\frac{\sinh\pi\mu}{\sinh\gamma\mu}\cosh(\phi+\phi_0-\gamma)\mu+\frac{\sinh(\pi-\gamma)\mu}{\sinh\gamma\mu}\cosh(\phi-\phi_0)\mu  
\end{equation*}

\begin{equation*}
  -\frac{\sinh\pi\mu}{\sinh\gamma\mu}
  \left(e^{(\phi+\phi_0-\gamma)\mu}\frac{\alpha}{\alpha+\beta\mu}+e^{-(\phi+\phi_0-\gamma)\mu}\frac{\alpha}{\alpha-\beta\mu}\right)\biggl\} d\mu.
\end{equation*}

The \textbf{Dirichlet} boundary condition corresponds to the Robin boundary condition with $\beta=0$.  In this case, the Green's function, 

\begin{equation} \label{DirichletGK}
  G_D(s,r,\phi,r_0,\phi_0)=\frac{1}{\pi^2}\int_{0}^{\infty}K_{i\mu}(r \sqrt s)K_{i\mu}(r_0 \sqrt s)
\end{equation}

\begin{equation*}
  \times\biggl\{ \cosh(\pi-|\phi_0-\phi|)\mu-\frac{\sinh\pi\mu}{\sinh\gamma\mu}\cosh(\phi+\phi_0-\gamma)\mu+\frac{\sinh(\pi-\gamma)\mu}{\sinh\gamma\mu}\cosh(\phi-\phi_0) \mu \biggl\}d\mu.
\end{equation*}

The \textbf{Neumann} boundary condition corresponds to the Robin boundary condition with $\alpha=0$.  In this case, the Green's function, 

\begin{equation} \label{NeumannGK} 
  G_N(s,r,\phi,r_0,\phi_0)=\frac{1}{\pi^2}\int_{0}^{\infty}K_{i\mu}(r \sqrt s)K_{i\mu}(r_0 \sqrt s)
\end{equation}

\begin{equation*}
  \times\biggl\{ \cosh(\pi-|\phi_0-\phi|)\mu+\frac{\sinh\pi\mu}{\sinh\gamma\mu}\cosh(\phi+\phi_0-\gamma)\mu+\frac{\sinh(\pi-\gamma)\mu}{\sinh\gamma\mu}\cosh(\phi-\phi_0) \mu  \biggl\}d\mu.
\end{equation*}

The details in the derivation of these formulas shall be presented in our forthcoming work.  The idea is to look for the Green's function, $G$, of the following form
\begin{equation}\label{G_est}
  G(s,r,\phi,r_0,\phi_0)=\frac{2}{\pi^2}\int_{0}^{\infty}K_{i\mu}(r\sqrt{s})K_{i\mu}(r_0\sqrt{s})\mu\sinh(\pi\mu)\Phi(\mu,\phi,\phi_0)d\mu.
\end{equation}
Above, $K_{i \mu}$ is the modified Bessel function of the second kind, and $s$ is the spectral parameter of the resolvent, $(\Delta + s)^{-1}$.  

In that work, we also use functional calculus techniques to rigorously justify the statement, 
$$H(t, r, \phi, r_0, \phi_0) = \mathcal{L}^{-1} \left( G(s, r, \phi, r_0, \phi_0) \right) (t).$$
Above, $H$ denotes the heat kernel, and $\mathcal L^{-1}$ denotes the inverse Laplace transform taken with respect to the spectral parameter, $s$.  This allows us to pass from the Green's functions to the heat kernels on a sector, and we may then compute the short time asymptotic expansions of the heat traces for all three boundary conditions.  We do this by explicitly calculating: 
\begin{enumerate} 
\item the integral of $H_{\R^2} (t,z,z)$ in the interior 
\item the integral of $H_{\R^2 _+} (t,z,z)$ near the edges but away from the corners, with the appropriate boundary condition,
\item the integral of $H_{S_\alpha} (t,z,z)$, where $H_{S_\alpha}$ is the heat kernel for an infinite circular sector of opening angle $\alpha$ with the same boundary condition as that taken on the polygon, near a corner of interior angle $\alpha$.
\end{enumerate}
Albeit rather lengthy and technical, those calculations may be of independent interest and shall be presented in our forthcoming work.  

In this way, we compute that for a polygonal domain $\Omega \subset \R^2$, with $n$ vertices, having interior angles $\theta_1,\dots,\theta_n$, the heat trace asymptotics for small times are respectively:  
\begin{enumerate}
\item[(D)] for the Dirichlet boundary condition, 
$$\tr e^{-t \Delta} \sim \frac{|\Omega|}{4 \pi t} - \frac{|\pa \Omega|}{8 \sqrt{\pi t}} + \frac{\chi(\Omega)}{6} - \frac n {12} + \sum_{k=1} ^n \frac{\pi^2 + \theta_k^2}{24\pi \theta_k} + O(\sqrt t ),$$
\item[(N)] for the Neumann boundary condition, 
$$\tr e^{-t \Delta} \sim \frac{|\Omega|}{4 \pi t} + \frac{|\pa \Omega|}{8 \sqrt{\pi t}} + \frac{\chi(\Omega)}{6} - \frac n {12} + \sum_{k=1} ^n \frac{\pi^2 + \theta_k^2}{24\pi \theta_k}+ O(\sqrt t ),$$
\item[(R)] for the Robin boundary condition, 
$$\tr e^{-t \Delta} \sim \frac{|\Omega|}{4 \pi t} + \frac{|\pa \Omega|}{8 \sqrt{\pi t}} + \frac{\chi(\Omega)}{6} - \frac{|\pa \Omega| \alpha}{2\pi \beta} - \frac n {12} + \sum_{k=1} ^n \frac{\pi^2 + \theta_k^2}{24\pi \theta_k}+ O(\sqrt t ).$$
\end{enumerate} 

Now let $\tilde \Omega$ be a smoothly bounded domain in the plane.  The heat trace expansions have been computed by \cite{ms} for the Dirichlet and Neumann boundary condition and \cite{zayed} for the Robin condition.  
These are, respectively,   
\begin{enumerate}
\item[(D)] for the Dirichlet boundary condition, 
$$\tr e^{-t \Delta} \sim \frac{|\tilde \Omega|}{4 \pi t} - \frac{|\pa \tilde \Omega|}{8 \sqrt{\pi t}} + \frac{\chi(\tilde \Omega)}{6} + O(\sqrt t ),$$
\item[(N)] for the Neumann boundary condition, 
$$\tr e^{-t \Delta} \sim \frac{|\tilde \Omega|}{4 \pi t} + \frac{|\pa \tilde \Omega|}{8 \sqrt{\pi t}} + \frac{\chi(\tilde \Omega)}{6} +  O(\sqrt t ),$$
\item[(R)] for R boundary condition, 
$$\tr e^{-t \Delta} \sim \frac{|\tilde \Omega|}{4 \pi t} + \frac{|\pa \tilde \Omega|}{8 \sqrt{\pi t}} + \frac{\chi(\tilde \Omega)}{6} - \frac{|\pa \tilde \Omega| \alpha}{2\pi \beta} + O(\sqrt t ).$$
\end{enumerate}
The above expressions coincide with the corresponding expressions for polygonal domains in the case $n=0$. Now, one may compare the constant terms in each of these expansions to the constant terms with the corresponding boundary condition in case $\Omega$ is a polygonal domain.  In so doing, it is a mere exercise in multivariable calculus to prove that these will never be equal for a simply connected polygonal domain, $\Omega$, and for any smoothly bounded domain $\tilde \Omega$.  This argument may be found in \cite{corners}.

\section{Microlocal analysis in the curvilinear case}  \label{microloc}

It turns out that the heat trace expansions above are also valid for curvilinear polygons, once terms accounting for the curvature of the boundary away from the corners have been added.  Although this has been demonstrated in \cite{corners} for the Dirichlet boundary condition using monotonicity, it becomes a much more subtle matter for the Neumann and Robin boundary conditions.  

The main problem is that for curvilinear polygons, we no longer have an exact geometric match.  Hence, we can no longer use the locality principle to compute the heat trace expansion.  For classical polygons, one may compute the Neumann heat trace using the Dirichlet heat trace together with the trace of a Euclidean surface with conical singularities created by doubling the polygon.  However, this technique fails once the edges of the polygon are no longer necessarily straight. Therefore, in order to compute the short time asymptotic expansion of the heat trace without exact geometric matches, we turn to the robust techniques of geometric microlocal analysis. This allows us to give a full description of the Dirichlet, Neumann, and Robin heat kernels on a curvilinear polygon in all asymptotic regimes. Restricting to the diagonal and integrating yields the heat trace.

In order to describe the heat kernel in all asymptotic regimes, we build a space, called the \emph{heat space} or \emph{double heat space}, on which the heat kernel is well-behaved. This space is built by blowing up various p-submanifolds of $\Omega \times \Omega \times [0, \infty)$. To see why this is needed, first consider the heat kernel \eqref{hkrn} on $\R^n$.  At the diagonal in $\R^n \times \R^n \times [0, \infty)$, the heat kernel behaves as $O(t^{-n/2})$ as $t \downarrow 0$.  However, as long as $d(z, z') \geq \eps > 0$, the heat kernel behaves as $O(t^\infty)$ as $t \downarrow 0$. So the heat kernel fails to be well-behaved at $\{z=z',t=0\}$.  This is the motivation for ``blowing up'' the diagonal $\{z = z'\}$ at $t=0$, which means replacing this diagonal with its inward pointing spherical normal bundle, corresponding to the introduction of ``polar coordinates". The precise meaning of ``blowing up'' is explained in \cite{tapsit}, and in this particular case of blowing up $\{ z = z' \}$ at $t=0$ in $\R^n \times \R^n \times [0, \infty)$, see \cite[Chapter 7]{tapsit}.

For the case of a curvilinear polygonal domain $\Omega \subset \R^2$,  we begin with $\Omega \times \Omega \times [0, \infty)$ and perform a sequence of blow-ups. Our construction is inspired by the construction of the heat kernel on manifolds with wedge singularities performed by Mazzeo and Vertman in \cite{mave}. We leave the details to our forthcoming work.

Once the double heat space has been constructed, the heat kernel may be built in a similar spirit to the Duhamel's principle construction of the Robin heat kernel in the proof of Theorem \ref{loc_pr_for_R}. We start with a parametrix, or initial guess, and then use Duhamel's principle to iterate away the error. This requires the proof of a composition result for operators whose kernels are well-behaved on our double heat space, and that in turn requires some fairly involved technical machinery (a proof ``by hand" without using this machinery would be entirely unreadable). However, it works out and gives us a very precise description of the heat kernel on a curvilinear polygon, with any combination of Dirichlet, Neumann, and Robin conditions.

The details of this sort of geometric microlocal analysis construction are intricate, but its utility is undeniable.  In settings such as this, where exact geometric matches are lacking, but instead, one has \em asymptotic geometric matches, \em these microlocal techniques may be helpful.  For the full story in the case of curvilinear polygons and their heat kernels, please stay tuned for our forthcoming work, in which we shall use the microlocal construction described above to prove Theorem 1.  


%
\begin{acknowledgement}
The authors extend their sincere gratitude to the organizers of the Matrix workshop, ``Elliptic Partial Differential Equations of Second Order: Celebrating 40 Years of Gilbarg and Trudinger's Book."  Many thanks to Lucio Boccardo, Florica-Corina C\^irstea, Julie Clutterbuck, L. Craig Evans, Enrico Valdinoci, Paul Bryan, and of course, Neil Trudinger!   The second author is also grateful to: Lashi Bandara, Feng-Yu Wang, Alexander Strohmaier, Liangpan Li and, Grigori Rozenblum for helpful correspondences. The first author was partially  supported by MON RK   AP05132071.
\end{acknowledgement}
%

%

\begin{thebibliography}{99.}%
%
%
\bibitem{polyakov} C. Aldana, J. Rowlett, \em A Polyakov formula for sectors, \em  J Geom Anal (2017) \url{https://doi.org/10.1007/s12220-017-9888-y} 

\bibitem{d} D.Daners, \em Heat kernel estimates for operators with boundary conditions, \em Mathematische Nachrichten, vol 217, no. 1, (2000), 13--41.

\bibitem{durso} C. Durso, \em On the inverse spectral problem for polygonal domains,  \em Ph. D. dissertation, 
Massachusetts Institute of Technology, Cambridge, 1988. 

\bibitem{gww1} C. Gordon, D. L. Webb, S. Wolpert, \em One cannot hear the shape of a drum, \em 
Bull. Amer. Math. Soc. (N.S.)  \textbf{27} no. 1 (1992) 134--138.  

\bibitem{gww2} \textemdash, \textemdash, \textemdash, \em Isospectral plane domains and surfaces via Riemannian orbifolds, \em Invent. Math.  \textbf{110} no. 1 (1992) 1--22.  

\bibitem{gm} D. Grieser, S. Maronna, \em Hearing the shape of a triangle, \em Notices Amer. Math. Soc.  \textbf{60} no. 11 (2013) 1440--1447.  

\bibitem{hlr} H. Hezari, Z. Lu, and J. Rowlett, \em The Neumann Isospectral Problem for Trapezoids, \em  Ann. Henri Poincar\'e 18, no. 12,  (2017), 3759--3792. 

\bibitem{kac} M. Kac, \em Can one hear the shape of a drum? \em  Amer. Math. Monthly  \textbf{73} no. 1 (1966) 1--23. 

\bibitem{list} L. Li \& A. Strohmaier, \em Heat kernel estimates for general boundary problems, \em J. Spectral Theory, 6, (2016) 903--919, (arxiv:1604.00784v1).

\bibitem{sos} Z. Lu and J. Rowlett, \em The sound of symmetry, \em  Amer. Math. Monthly 122, no. 9,  (2015), 815--835.

\bibitem{corners} Z. Lu and J. Rowlett, \em One can hear the corners of a drum, \em Bull. Lond. Math. Soc. 48, no. 1, (2016), 85--93.

\bibitem {luck} W. L\"uck and T. Schick, \em $L^2$-torsion of hyperbolic manifolds of finite volume, \em Geom. Funct. Anal., 9, no. 3, (1999), 518--567.

\bibitem{mave} R. Mazzeo and B. Vertman, \emph{Analytic torsion on manifolds with edges.} Adv. Math. 231, no. 2, (2012), 1000--1040.

\bibitem{tapsit} R. Melrose, \em The Atiyah-Patodi-Singer Index Theorem, \em Research Notes in Mathematics 4. A K Peters, Ltd., (1993).

\bibitem{ms} H. P. McKean Jr., I. M. Singer, \em Curvature and the eigenvalues of the Laplacian, \em 
J. Diff. Geom.  \textbf{1} no. 1 (1967) 43--69.  

\bibitem{milnor} J. Milnor, \em Eigenvalues of the Laplace operator on certain manifolds, \em Proc. Nat. Acad. Sci. U.S.A.  \textbf{51} (1964) 542.  

\bibitem{P} V. G. Papanicolaou, \em The probabilistic solution of the third boundary value problem for second order elliptic equations, \em  Probab. Theory Related Fields, 87, no. 1, (1990), 27--77.


\bibitem{pl} {\AA}. Pleijel,  \em A study of certain Green's functions with applications in the
   theory of vibrating membranes, \em Ark. Mat.  \textbf{2} (1954) 553--569.

\bibitem{sherconicdegen} D. Sher, \em Conic degeneration and the determinant of the Laplacian, \em J. Anal. Math. \textbf{126} no. 2 (2015) 175--226.
   
   \bibitem{sun} T. Sunada, \em Riemannian coverings and isospectral manifolds, \em Ann. of Math. (2) \textbf{121} no. 1 (1985) 169--186.  
   
   \bibitem{vdbs} M. van den Berg,  S. Srisatkunarajah,  \em Heat equation for a region in $\R^2$ with polygonal boundary, \em J. London Math. Soc. (2)  2 no. 1 (1988) 119--127.  

\bibitem{taylor} M. Taylor, \em Partial differential equations I, Basic theory, \em Applied Math. Sciences, vol. 115, second edition, Springer, New York, (2011).  

\bibitem{wang1} F.Y. Wang \& L. Yan, \em Gradient estimate on the Neumann semigroup and applications, \em arxiv:1009.1965v2.  

\bibitem{wang2} F. Y. Wang, \em Gradient estimates and the first Neumann eigenvalue on manifolds with boundary, \em Stoch. Proc. Appl. 115 (2005), 1475--1486.

\bibitem{wang3} F. Y. Wang, \em Semigroup properties for the second fundamental form, \em Docum. Math. 15, (2010), 543--559.

\bibitem{wang4} F. Y. Wang, \em Gradient and Harnack inequalities on noncompact manifolds with boundary, \em Pacific Journal of Math. 245, (2010), 185--200.

\bibitem{wang5} F. Y. Wang, \em private communication.\em

\bibitem{weyl} H. Weyl, \em Das asymptotische Verteilungsgesetz der Eigenwerte linearer
   partieller Differentialgleichungen (mit einer Anwendung auf die Theorie
   der Hohlraumstrahlung), \em Math. Ann.  \textbf{71} no. 4 (1912) 441--479.  

\bibitem{witt} E. Witt, \em Eine Identit\"at zwischen Modulformen zweiten Grades, \em Abh. Math. Sem. Univ. Hamburg,
 \textbf{14} no. 1 (1941) 323--337.  
 
 \bibitem{zayed}   E. M. Zayed, \em Short time asymptotics of the heat kernel of the Laplacian of a bounded domain with Robin boundary conditions, \em Houston J. Math, vol. 24, no. 2, (1998), 377--385.

\bibitem{zelda} S. Zelditch, \em Spectral determination of analytic bi-axisymmetric plane domains, \em Geom. Funct. Anal. 10, no. 3,  (2000), 628--677. 
%



\end{thebibliography}
%

\end{document}